\documentclass[11pt]{article}

\usepackage{amsmath}
\usepackage{amsfonts}
\usepackage{color}
\usepackage{amssymb}
\usepackage{graphicx}
\usepackage{hyperref,arydshln}
\usepackage{enumerate}
\usepackage[all]{xy}

\def\tr{{\rm tr}}

 \allowdisplaybreaks

\begin{document}

\newtheorem{problem}{Problem}

\newtheorem{theorem}{Theorem}[section]
\newtheorem{corollary}[theorem]{Corollary}
\newtheorem{definition}[theorem]{Definition}
\newtheorem{conjecture}[theorem]{Conjecture}
\newtheorem{question}[theorem]{Question}
\newtheorem{lemma}[theorem]{Lemma}
\newtheorem{proposition}[theorem]{Proposition}
\newtheorem{quest}[theorem]{Question}
\newtheorem{example}[theorem]{Example}

\newenvironment{proof}{\noindent {\bf
Proof.}}{\rule{2mm}{2mm}\par\medskip}

\newenvironment{proofof3}{\noindent {\bf
Proof of  Theorem 1.2.}}{\rule{2mm}{2mm}\par\medskip}

\newenvironment{proofof5}{\noindent {\bf
Proof of  Theorem 1.3.}}{\rule{2mm}{2mm}\par\medskip}

\newcommand{\remark}{\medskip\par\noindent {\bf Remark.~~}}
\newcommand{\pp}{{\it p.}}
\newcommand{\de}{\em}

\title{  {A new matrix inequality involving partial traces}\thanks{
This paper was firstly announced on Feb, 2020, and was later published on Operators and Matrices, 15 (2021), no.3, 1189--1199. 
See \url{http://oam.ele-math.com/15-75}. 
 E-mail addresses: 
ytli0921@hnu.edu.cn (Y. Li), 
wjliu6210@126.com (W. Liu), 
FairyHuang@csu.edu.cn (Y. Huang, corresponding author).} }
\date{June 9, 2021}

\author{Yongtao Li$^{a}$,  Weijun Liu$^b$, Yang Huang$^{a,\dag}$\\
{\small ${}^a$School of Mathematics, Hunan University} \\
{\small Changsha, Hunan, 410082, P.R. China } \\
{\small $^b$School of Mathematics and Statistics, Central South University} \\
{\small New Campus, Changsha, Hunan, 410083, P.R. China. } }

\maketitle

\vspace{-0.5cm}

\begin{abstract}
Let $A$ be an $m\times m$ positive semidefinite  block matrix with each block being $n$-square.  
We write $\mathrm{tr}_1$ and $\mathrm{tr}_2$  for the first and second partial trace, respectively. 
In this paper, we prove the following inequality 
\[ 
(\tr A)I_{mn}  - (\tr_2 A) \otimes I_n
\ge \pm  \bigl( I_m\otimes (\tr_1 A)  -A\bigr).
\] 
This inequality is not only  a generalization of 
Ando's result [ILAS Conference (2014)] and Lin [Canad. Math. Bull. 59 (2016) 585--591], but it also  
could be regarded as a complement of a recent result of  
Choi [Linear Multilinear Algebra 66 (2018) 1619--1625]. 
Additionally, some new partial traces inequalities for positive semidefinite block matrices  
are also included. 
 \end{abstract}



\section{Introduction}

\label{sec1} 

We use the following standard notation; see, e.g., \cite{Bha97} and \cite{HJ13}. 
The set of $n\times n$ complex matrices is denoted by $\mathbb{M}_n(\mathbb{C})$, 
or simply by $\mathbb{M}_n$, 
and the identity matrix of order $n$ by  $I_n$, or $I$ for short. 
If $A=[a_{ij}]$ is of order $m\times n$ and  
 $B$ is of order $s\times t$, the {\it tensor product} of $A$ with $B$, 
denoted by $A\otimes B$, which is an $ms\times nt$ matrix that  
partitioned into $m\times n$ block matrix with the $(i,j)$-block being the $s\times t$ matrix $a_{ij}B$. 
By convention, if $X\in \mathbb{M}_n$ is positive semidefinite, then we  write $X\ge 0$. 
For two Hermitian matrices $A$ and $B$ of the same order, $A\ge B$ stands for $A-B\ge 0$; 
see \cite[Chapter 1]{Zhan02} and \cite{Zhan13}.  
In this paper, 
we are interested in complex block matrices. Let $\mathbb{M}_m(\mathbb{M}_n)$ 
be the set of complex matrices partitioned into $m\times m$ blocks 
with each block being an $n\times n$ matrix. 
The element of $\mathbb{M}_m(\mathbb{M}_n)$ is usually written as ${ A}=[A_{i,j}]_{i,j=1}^m$ 
with $A_{i,j}\in \mathbb{M}_n$ for every $1\le i,j\le m$.

For $A=[A_{i,j}]_{i,j=1}^m\in \mathbb{M}_m(\mathbb{M}_n)$, 
we define the {\it partial transpose} of $A$ by $A^{\tau}=[A_{j,i}]_{i,j=1}^m$. 
It is clear that $A\ge 0$ does not necessarily imply $A^{\tau}\ge 0$. 
For instance, taking 
\[ A=\begin{bmatrix} 
A_{1,1} & A_{1,2} \\ A_{2,1} & A_{2,2} \end{bmatrix}=
\left[\begin{array}{cc;{2pt/2pt}cc}
1&0&0&1\\
0&0&0&0\\
\hdashline[2pt/2pt] 
0&0&0&0\\
1&0&0&1 \end{array}\right] \ge 0.\]
It follows by definition that  
\[ A^{\tau}=\begin{bmatrix} 
A_{1,1} & A_{2,1} \\ A_{1,2} & A_{2,2} \end{bmatrix}=
\left[\begin{array}{cc;{2pt/2pt}cc}
1&0&0&0\\
0&0&1&0\\
\hdashline[2pt/2pt] 
0&1&0&0\\
0&0&0&1 \end{array}\right]. \]
One could easily observe that $A^{\tau}$ is not positive semidefinite 
since it contains a principal submatrix $
\left[\begin{smallmatrix}0 & 1 \\ 1 &0 \end{smallmatrix}\right] \ngeq 0$. 
If both $A$ and $A^{\tau}$ are positive semidefinite, 
then $A$ is said to be {\it positive partial transpose} or PPT for short; see \cite{Lin14,Lin15,Lee15}. 
For more explanations of the partial transpose and  PPT, 
we recommend a comprehensive survey \cite{Bh07}, 
and see, e.g., \cite{Choi17,Choi532,Choi17b,Zhang19} for more recent results.

Now we introduce the definition and notation of partial traces, 
which comes from Quantum Information Theory \cite[pp. 10--12]{Petz08}.  
For $A\in \mathbb{M}_m(\mathbb{M}_n)$, 
the first partial trace map $A \mapsto \mathrm{tr}_1 A \in \mathbb{M}_n$ is defined as the  
adjoint map of the imbedding map $X \mapsto I_m\otimes X\in \mathbb{M}_m\otimes \mathbb{M}_n$. 
Correspondingly, the second partial trace map  $A \mapsto \mathrm{tr}_2 A\in \mathbb{M}_m$ is 
similarly defined as the adjoint map of the imbedding map 
$Y\mapsto Y\otimes I_n \in \mathbb{M}_m\otimes \mathbb{M}_n$. Therefore, we have 
\begin{equation*} \label{eqdef} 
\langle I_m\otimes X, A \rangle =\langle X, \mathrm{tr}_1A \rangle ,
\quad \forall X\in \mathbb{M}_n, 
\end{equation*}
and 
\[ \langle Y\otimes I_n, A \rangle =\langle Y,\mathrm{tr}_2 A \rangle, 
\quad \forall Y\in \mathbb{M}_m. \]
Assume that $A=[A_{i,j}]_{i,j=1}^m$ with $A_{i,j}\in \mathbb{M}_n$, 
  equivalent forms of the first and second partial trace 
are  given in  \cite[pp. 120--123]{Bh07} as
\begin{equation*} \label{eqdef2}
 \mathrm{tr}_1 { A}=\sum\limits_{i=1}^m A_{i,i} ~~\text{and}~~ 
\mathrm{tr}_2{ A}=\bigl[ \mathrm{tr}A_{i,j}\bigr]_{i,j=1}^m. 
\end{equation*}

As we all know, these two partial traces maps are linear and trace-preserving. 
Furthermore, if ${ A}=[A_{i,j}]_{i,j=1}^m \in \mathbb{M}_m(\mathbb{M}_n)$ is positive semidefinite,  
it is easy to see that  both $\mathrm{tr}_1 A$ and $\mathrm{tr}_2 A$ 
are positive semidefinite; see, e.g., \cite[p. 237]{Zhang11} or \cite[Theorem 2.1]{Zha12}. 
To some extent, 
these two partial traces are closely related.  
For instance, Ando \cite{Ando14}  established  

\begin{equation*}
(\tr A)I_{mn} +  A \ge  
I_m\otimes (\mathrm{tr}_1 A)  + (\mathrm{tr}_2 A) \otimes I_n.
\end{equation*}
We refer to \cite{Lin16} for an alternative proof. Equivalently, it can be written as 
\begin{equation}  \label{eqando}
 (\tr A)I_{mn} - (\mathrm{tr}_2 A) \otimes I_n \ge  
I_m\otimes (\mathrm{tr}_1 A)   -  A.   
\end{equation}
Moreover, Choi  recently investigated the first partial trace in \cite{Choi17}  and presented  
\[ I_m\otimes \tr_1 A^{\tau} \ge  A^{\tau},\] 
Meanwhile, Choi also proved in \cite{Choi17} that if $A\in \mathbb{M}_2(\mathbb{M}_n)$ is positive semidefinite, then 
\[ I_2 \otimes (\tr_1 A) + (\tr_2 A) \otimes I_n \ge A. \] 
Furthermore, Choi \cite{Choi17b} gave  a further extension and showed 
\begin{equation} \label{eqchoi}
\text{$(\tr_2 A^{\tau})\otimes I_n\ge \pm A^{\tau}$ ~~and~~ $ 
I_m\otimes \tr_1 A^{\tau} \ge \pm A^{\tau}$.}
\end{equation}

We observe in (\ref{eqando}) that the positivity of  $A$ leads to   
\[ (\tr A)I_m =\sum_{i=1}^m (\tr A_{i,i})I_m =
\bigl( \tr (\tr_2 A)\bigr)I_m \ge \tr_2 A,\]
 which guarantees that 
$(\tr A)I_{mn} - (\mathrm{tr}_2 A) \otimes I_n $ is positive semidefinite. 
However, the two matrices of right hand side in (\ref{eqando}) might be incomparable. 
A PPT condition  on block matrix $A$ was proposed 
to ensure $I_m\otimes (\mathrm{tr}_1 A)  \ge A$; 
see \cite{Choi17b} or   \cite[Corollary 2.2]{Li20laa} for more details. 

As we have already discussed above, and motivated by Choi's result (\ref{eqchoi}), 
 we will give a new partial traces inequality  (Theorem \ref{thm22}), 
which   could be viewed as  a generalization of Ando's result (\ref{eqando}) 
and also a complement of Choi's result (\ref{eqchoi}). 

The paper is organized as follows. 
We first introduce an efficient and useful lemma, which was first proved by Lin \cite{Lin14}. 
We will provide an alternative short proof of this lemma for completeness   
 and then utilize it to prove Theorem \ref{thm22}.  
 Additionally,  
 we present some new partial traces inequalities 
 (Theorem \ref{thm25} and Corollary \ref{coro26})
for positive semidefinite block matrices. 
As an application on numerical analysis,  
we  give some generalizations of the famous Cauchy-Khinchin inequality (Corollary \ref{coro24} and \ref{coro25}).

\section{Main result}
\label{sec2}

A  map (not necessarily linear) $\Phi: \mathbb{M}_n\to \mathbb{M}_k$ is called positive 
if it maps positive semidefinite matrices 
to positive semidefinite matrices. 
A  map $\Phi: \mathbb{M}_n\to \mathbb{M}_k$  is said to be {\it $m$-positive} if 
for every $[A_{i,j}]_{i,j=1}^m\in \mathbb{M}_m(\mathbb{M}_n)$,  
\begin{equation}  \label{eq1}
[A_{i,j}]_{i,j=1}^m \ge 0 \Rightarrow [\Phi (A_{i,j})]_{i,j=1}^m\ge 0. 
\end{equation}
The   map $\Phi$ is said to be {\it completely positive} 
if (\ref{eq1}) holds for every positive integer $m\ge 1$. 
It is well-known that both the trace map and  determinant map 
are  completely positive; see, e.g., \cite[p. 221, p. 237]{Zhang11} 
and \cite{Zha12}. 
On the other hand, a  map  $\Phi $  is said to be {\it $m$-copositive} if 
for every $[A_{i,j}]_{i,j=1}^m\in \mathbb{M}_m(\mathbb{M}_n)$, 
\begin{equation}  \label{eq2}
[A_{i,j}]_{i,j=1}^m \ge 0 \Rightarrow [\Phi (A_{j,i})]_{i,j=1}^m\ge 0,  
\end{equation}
and $\Phi$ is said to be {\it completely copositive} 
if (\ref{eq2}) holds for every positive integer $m\ge 1$. 
Furthermore, 
a map $\Phi$ is called {\it a completely PPT} if it is both completely positive and completely copositive. 
A comprehensive survey on completely positive maps can be found in 
\cite[Chapter 3]{Bh07}.

Before starting our proof of Theorem \ref{thm22}, 
 we first introduce the following useful Lemma \ref{lem21}, 
which is not only  the main result in \cite[Theorem 1.1]{Lin14}, 
but also plays an important role in our proof.  
We here provide an alternative proof for completeness; 
see \cite{Li20laa} for more potential applications 
and \cite{FLTmia} for the relation 
with singular value inequality.

\begin{lemma}  \cite{Lin14} \label{lem21}
The map $\Phi (X)= X + (\tr X) I$ is completely PPT.
\end{lemma}

\begin{proof}
We use the Choi's criterion \cite{Choi72} to give a short proof.  
This criterion  is now becoming  a standard tool for completely PPT map in 
quantum information theory. 
It suffices to prove that for every positive integer $m$, 
\[ [\Phi (E_{j,i})]_{i,j=1}^m \ge 0, \]
where $E_{j,i}\in \mathbb{M}_n$ stands for the unit matrix, that is, the matrix with 
$1$ in the $(j,i)$-th entry and $0$ elsewhere.  
Note that $[\Phi (E_{j,i})]_{i,j=1}^m$ is symmetric and row diagonally dominant with nonnegative 
diagonal entries. Then $[\Phi (E_{j,i})]_{i,j=1}^m$  is positive semidefinite for each $m$. 
So $[\Phi (A_{j,i})]_{i,j=1}^m$ is positive semidefinite.
On the other hand, let $A=[A_{i,j}]_{i,j=1}^m$ be positive semidefinite. 
Since $[\tr A_{i,j}]_{i,j=1}^m$ is positive semidefinite \cite[p. 237]{Zhang11} and 
\[ [\Phi (A_{i,j})]_{i,j=1}^m = [\tr A_{i,j}]_{i,j=1}^m \otimes I_n + A,  \] 
then  $[\Phi (A_{i,j})]_{i,j=1}^m$ is also positive semidefinite. This completes the proof.
\end{proof}

Now, we are ready to present the main result. 
Our result could be viewed as a 
generalization and complement of both (\ref{eqando}) and (\ref{eqchoi}). 

  \begin{theorem}\label{thm22} 
Let $A=[A_{i,j}]_{i,j=1}^m\in \mathbb{M}_m(\mathbb{M}_n)$ be positive semidefinite. Then
\[  
(\tr A)I_{mn}  - (\tr_2 A) \otimes I_n
\ge \pm  \bigl( I_m\otimes (\tr_1 A)  -A\bigr).
\] 
 \end{theorem}

\begin{proof}
As Ando's result (\ref{eqando}), we only need to prove that 
  	\begin{equation}\label{eqmain} 
(\tr A)I_{mn} - (\tr_2 A) \otimes I_n   
\ge  A -I_m\otimes (\tr_1 A) .
  	\end{equation}  
When $m=1$, there is nothing to prove. 
We now prove the case   $m=2$. In this case, the required inequality is  
\begin{align*}  
& \begin{bmatrix}  (\tr A )I_n& 0\\ 0& (\tr A)I_n \end{bmatrix} - 
\begin{bmatrix} (\tr A_{1,1})I_n& (\tr A_{1,2})I_n\\ (\tr A_{2,1})I_n & (\tr A_{2,2})I_n \end{bmatrix} 
  \\
 &\quad \ge  \begin{bmatrix} A_{1,1}& A_{1,2}\\ A_{2,1} & A_{2,2} \end{bmatrix} - 
\begin{bmatrix} A_{1,1}+A_{2,2}& 0\\ 	0& A_{1,1}+A_{2,2}\end{bmatrix}, 
 \end{align*}
 or equivalently (note that $\tr A=\tr A_{1,1} +\tr A_{2,2}$),	
\begin{equation} \label{eqbs}
  M:=\begin{bmatrix}
  		(\tr A_{2,2})I_n+ A_{2,2}&   -A_{1,2}-(\tr A_{1,2})I_n\\  
-A_{2,1}-(\tr A_{2,1})I_n& 	(\tr A_{1,1})I_n+ A_{1,1}
  	\end{bmatrix}\ge 0.  
\end{equation}
By Lemma \ref{lem21}, we get 
 			\begin{eqnarray*} \begin{bmatrix}
  					(\tr A_{1,1})I_n+ A_{1,1}& (\tr A_{2,1})I_n+ A_{2,1}\\ 	
(\tr A_{1,2})I_n+ A_{1,2}& 	(\tr A_{2,2})I_n+ A_{2,2}
  				\end{bmatrix}\ge 0,  \end{eqnarray*} 
and so 
  				\begin{eqnarray*} 
M=\begin{bmatrix}0 & -I_n\\  I_n& 0\end{bmatrix} 
\begin{bmatrix} (\tr A_{1,1})I_n+ A_{1,1}& (\tr A_{2,1})I_n+ A_{2,1}\\ 
(\tr A_{1,2})I_n+ A_{1,2}& (\tr A_{2,2})I_n+ A_{2,2} \end{bmatrix} 
\begin{bmatrix}0 &  I_n\\  -I_n& 0\end{bmatrix} \ge 0,  \end{eqnarray*} 
which confirms the desired  (\ref{eqbs}).

Next, we turn to the general case. 
Our treatment in this case has its root in \cite{Ando14}. 
By definition, setting  
  		\begin{align*}  
\Gamma &:=(\tr A)I_{mn}+ I_m \otimes(\tr_1 A)-A-(\tr_2 A) \otimes I_n	\\
&= \left(\tr \sum_{i=1}^{m}A_{i,i}\right) I_{mn}
+I_m\otimes \left(\sum_{j=1}^m A_{j,j}\right)-A-\bigl( [\tr A_{j,k}]_{j,k=1}^m\bigr) \otimes I_n \\
&= \left[ \delta_{j,k}\Bigl( \sum_{i=1}^m \tr A_{ii}\Bigr)I_n + 
\delta_{j,k} \Bigl(\sum_{i=1}^m A_{i,i} \Bigr) -A_{j,k}- (\tr A_{j,k})I_n  \right]_{j,k=1}^m. 
\end{align*}
For each pair $(p,q)$ with $1\le p<q \le m$, 
we define a $2\times m$ matrix $I_{p,q}$ as 
\[ 
I_{p,q}:=[\delta_{j,1}\delta_{k,p} +\delta_{j,2}\delta_{k,q}]_{j,k=1}^{2,m} 
= \left[\begin{array}{ccccccccccc}
0 & \cdots & 0 & \overset{ \text{$p$-th} }{1} & 0 & 
\cdots & 0& 0 &0 &\cdots & 0 \\
0 & \cdots & 0 & 0 & 0 & \cdots & 0& 
\overset{ \text{$q$-th} }{1}
&0 &\cdots & 0 
\end{array} \right].   \]
Upon a direct computation, it follows that 
\[ \Gamma = \sum_{1\le p<q \le m} (I_{p,q}\otimes I_n)^* M_{p,q} (I_{p,q}\otimes I_n), \]
where $M_{p,q}\in \mathbb{M}_2(\mathbb{M}_n)$ are defined as 
\[  M_{p,q}:=\begin{bmatrix}
(\tr A_{q,q})I_n +A_{q,q} & -A_{p,q}-(\tr A_{p,q})I_n \\
-A_{q,p} -(\tr A_{q,p})I_n & (\tr A_{p,p})I_n +A_{p,p} \end{bmatrix}.\]
It is easy to see from the case $m=2$ that the positivity of 
$\begin{bmatrix} A_{p,p} & A_{p,q} \\  A_{q,p} & A_{q,q}  \end{bmatrix}$ yields 
$M_{p,q}\ge 0$.  Hence, we get $\Gamma \ge 0$. 
This completes the proof. 
\end{proof}

Over the years, $2\times 2$ block positive semidefinite matrices are well studied, 
such a partition yields various elegant matrix inequalities; 
see \cite{Ando16,GLRT18,KL17,Lin15} for recent results. 
Next, we will give a partial traces inequality in the form of  $2\times 2$ block matrix. 

\begin{corollary}
Let $A=[A_{i,j}]_{i,j=1}^m\in \mathbb{M}_m(\mathbb{M}_n)$ be positive semidefinite. Then 
\begin{equation} \label{eq77}
\begin{bmatrix}
(\tr A)I_{mn} & A \\ 
A &  (\tr A)I_{mn}  \end{bmatrix} 
\ge 
\begin{bmatrix}
(\tr_2 A)\otimes I_n & I_m\otimes (\tr_1 A) \\ I_m\otimes (\tr_1 A) & (\tr_2 A)\otimes I_n   
\end{bmatrix}. 
\end{equation}
\end{corollary}

\begin{proof}
Note that 
\[ \begin{bmatrix} I & I \\ I & -I  \end{bmatrix} 
\begin{bmatrix} X & Y \\ Y & X  \end{bmatrix} \begin{bmatrix} I & I \\ I & -I  \end{bmatrix} 
=\begin{bmatrix} 2(X+Y) & 0 \\ 0 & 2(X-Y)  \end{bmatrix} \]
for any $X$ and $Y$ with same size. By this identity and Theorem \ref{thm22}, it follows that 
\[ \begin{bmatrix}
(\tr A)I_{mn} -(\tr_2 A)\otimes I_n & I_m\otimes (\tr_1 A)-A \\
I_m\otimes (\tr_1 A)-A &  (\tr A)I_{mn} -(\tr_2 A)\otimes I_n \end{bmatrix}\ge 0.  \]
By left and right-multiplying $\Bigl[ \!\begin{smallmatrix} I &0 \\ 0& -I
\end{smallmatrix}\! \Bigr]$, the disired result (\ref{eq77}) immediately holds. 
\end{proof}

We next provide an analogous result of Theorem \ref{thm22} under the PPT condition.  

  \begin{proposition}\label{prop23} 
Let $A=[A_{i,j}]_{i,j=1}^m\in \mathbb{M}_m(\mathbb{M}_n)$ be PPT. Then
\[ 
(\tr A)I_{mn}  + (\tr_2 A) \otimes I_n
\ge  I_m\otimes (\tr_1 A)  +A.
\] 
 \end{proposition} 

\begin{proof}
The required proposition holds from the following
\[  (\tr A)I_{mn} \ge I_m \otimes (\tr_1 A)\quad \text{~and~} \quad 
(\tr_2 A)\otimes I_n \ge A.  \]
The first inequality follows by 
\[ (\tr A)I_n =\sum_{i=1}^m(\tr A_{i,i})I_n 
\ge \sum_{i=1}^m A_{i,i}=\tr_1 A, \]  
and the second one is a direct consequence of 
Choi's result (\ref{eqchoi}). 
\end{proof}

At the end of this section,  
we  will provide more partial trace inequalities (Theorem \ref{thm25}) 
by using a similar approach as in \cite[Theorem 6]{Choi17b}. 
Let us start with some notation.  
Let $A=[A_{i,j}]_{i,j=1}^m\in \mathbb{M}_m(\mathbb{M}_n)$ 
and suppose that $A_{i,j}=\bigl[ a^{i,j}_{r,s}\bigr]_{r,s=1}^n$. 
We define $\widetilde{A}\in \mathbb{M}_n(\mathbb{M}_m)$ by 
\[ \text{ $ \widetilde{A} := [B_{r,s}]_{r,s=1}^n$, 
~~~where $B_{r,s}=\bigl[ a^{i,j}_{r,s}\bigr]_{i,j=1}^m\in \mathbb{M}_m$.}   \]
Clearly, we have $\widetilde{\widetilde{A}}=A$,  
and it was shown in \cite[Theorem 7]{Choi532} that $\widetilde{A}$ is unitarily similar with $A$.  
 This implies that 
if $A$ is positive semidefinite, then so is $\widetilde{A}$;  
see, e.g., \cite{Choi17,Choi17b} for more datails.
By a direct computation, we can see that 
\begin{equation}\label{eqtr1}
 \mathrm{tr}_2 \widetilde{A} =
\left[ \mathrm{tr} \bigl[a_{r,s}^{i,j}\bigr]_{i,j=1}^m \right]_{r,s=1}^n =
\left[ \begin{matrix} \sum\limits_{i=1}^m  a_{r,s}^{i,i}\end{matrix}  \right]_{r,s=1}^n=
\sum\limits_{i=1}^m \left[ a_{r,s}^{i,i} \right]_{r,s=1}^n=\mathrm{tr}_1A.  
\end{equation}
Moreover, for any $X=[x_{ij}]_{i,j=1}^m\in \mathbb{M}_m$ and $Y=[y_{rs}]_{r,s=1}^n\in \mathbb{M}_n$, by definition,  
\[ X\otimes Y=[x_{ij}Y]_{i,j=1}^m =\left[ [x_{ij}y_{rs}]_{r,s=1}^n \right]_{i,j=1}^m. \]
Then, it follows that  
\begin{equation} \label{eqtr2}
 \widetilde{X\otimes Y}=\left[ [x_{ij}y_{rs}]_{i,j=1}^m \right]_{r,s=1}^n
=\left[ y_{rs}X \right]_{r,s=1}^n=Y\otimes X.  
\end{equation}

  \begin{theorem}\label{thm25} 
Let $A=[A_{i,j}]_{i,j=1}^m\in \mathbb{M}_m(\mathbb{M}_n)$ be positive semidefinite. Then
\[   
(\tr A )I_{nm}  - (\tr_1 A) \otimes I_m
\ge \pm  \bigl( I_n\otimes (\tr_2 A)  -\widetilde{A}\bigr), \] 
and 
\[ 
(\tr A )I_{nm}  + (\tr_1 A) \otimes I_m
\ge  I_n\otimes (\tr_2 A)  +\widetilde{A}.
\] 
 \end{theorem}

\begin{proof}
Since $\widetilde{A} \in \mathbb{M}_n(\mathbb{M}_m)$, 
by applying Theorem \ref{thm22} to $\widetilde{A}$, we get 
\[ 
(\tr \widetilde{A} )I_{nm}  - (\tr_2 \widetilde{A}) \otimes I_m
\ge \pm  \bigl( I_n\otimes (\tr_1 \widetilde{A})  -\widetilde{A}\bigr),
\] 
Noth that $\tr \,\widetilde{A} =\tr \, A$ and combining (\ref{eqtr1}), it follows that 
\[ 
(\tr A )I_{nm}  - (\tr_1 A) \otimes I_m
\ge \pm  \bigl( I_n\otimes (\tr_2 A)  -\widetilde{A}\bigr).
\] 
On the other hand, by taking $\sim$ both sides in Theorem \ref{thm22}, we obtain  
\[ 
\widetilde{(\tr A)I_{mn} } -\widetilde{ (\tr_2 A) \otimes I_n }
\ge \pm  \bigl( \widetilde{I_m\otimes (\tr_1 A)}  -\widetilde{A}\bigr),
\] 
which together with (\ref{eqtr2}) leads to the following 
\[ (\tr A)I_{nm} - I_n \otimes (\tr_2 A) \ge 
\pm \bigl( (\tr_1 A)\otimes I_m -\widetilde{A}\bigr). \]
This completes the proof. 
\end{proof}

After finishing the first version of this paper, M. Lin 
suggested the author that an equivalent version of Theorem \ref{thm25} 
could be added as a corollary, which not only weakens the PPT condition in Proposition \ref{prop23}, 
but also can be regarded as  a complement of (\ref{eqmain}). 

\begin{corollary} \label{coro26}
Let $A=[A_{i,j}]_{i,j=1}^m\in \mathbb{M}_m(\mathbb{M}_n)$ be positive semidefinite. Then
\[ (\tr A)I_{mn} \pm (\tr_2 A)\otimes I_n \ge A \pm I_m \otimes (\tr_1 A). \]
Equivalently, it also could be written as 
\[  (\tr A)I_{mn} -A \ge \pm \bigl( I_m\otimes (\tr_1 A) -(\tr_2 A)\otimes I_n \bigr).  \]
\end{corollary}

\section{Applications}

As promised, 
we shall provide some applications of Theorem \ref{thm22} and Corollary \ref{coro26} 
in the field of numerical inequalities. 
The Cauchy-Khinchin inequality  
is well-known in the literature (see \cite[Theorem 1]{van98}), 
it  states that if $X=(x_{ij})$ is a real $m\times n$ matrix, then 
\begin{equation} \label{eqck} 
 \left(\sum_{i=1}^m\sum_{j=1}^nx_{ij}\right)^2+mn\sum_{i=1}^m\sum_{j=1}^nx_{ij}^2\ge m\sum_{i=1}^m\left(\sum_{j=1}^nx_{ij}\right)^2+n\sum_{j=1}^n\left(\sum_{i=1}^mx_{ij}\right)^2. 
\end{equation}

Next, we will give  a generallization and extension of (\ref{eqck})
 by using Theorem \ref{thm22} and Corollary \ref{coro26}, respectively; 
see, e.g., \cite{Lin16} for more determinantal inequalities.

  	\begin{corollary} \label{coro24}
 Let $X=(x_{ij})$ be a real $m\times n$ matrix. Then
  		\begin{eqnarray*}    
mn\sum_{i=1}^m\sum_{j=1}^nx_{ij}^2 - n\sum_{j=1}^n\left(\sum_{i=1}^mx_{ij}\right)^2
\ge \left| m\sum_{i=1}^m\left(\sum_{j=1}^nx_{ij}\right)^2  -
 \left(\sum_{i=1}^m\sum_{j=1}^nx_{ij}\right)^2 \right|. 
  		\end{eqnarray*}
  	\end{corollary}

  	\begin{proof} 
Let  $\mathrm{vec}\, X=[x_{11}, \ldots, x_{1n}, x_{21}, \ldots, x_{2n}, \ldots, x_{m1}, \ldots, x_{mn}]^T$ be a vectorization of $X$ and let $J_n$ be an $n$-square matrix with all entries 1.  
Then a simple calculation gives
  			\begin{eqnarray*} 
  	 (\mathrm{vec}\, X)^TI_{mn}(\mathrm{vec}\, X) &=&(\mathrm{vec}\, X)^T 
{\mathrm{vec}\, X}=\sum_{i=1}^m\sum_{j=1}^nx_{ij}^2,\\  		
  	 (\mathrm{vec}\, X)^T(I_m\otimes J_n)(\mathrm{vec}\, X)
   &=&\sum_{i=1}^m\left(\sum_{j=1}^nx_{ij}\right)^2, 	 \\
(\mathrm{vec}\, X)^T(J_m\otimes I_n)(\mathrm{vec}\, X) 
   &=&\sum_{j=1}^n\left(\sum_{i=1}^mx_{ij}\right)^2,\\ 
(\mathrm{vec}\, X)^T (J_m\otimes J_n)(\mathrm{vec}\, X) 
&=&(\mathrm{vec}\, X)^T  J_{mn} (\mathrm{vec}\, X)=\left(\sum_{i=1}^m\sum_{j=1}^nx_{ij}\right)^2.
  			\end{eqnarray*}
  		Thus the  desired inequality is equivalent to
  	\begin{equation} \label{ck} \begin{aligned}  
&(\mathrm{vec}\, X)^T ( mnI_{mn}  -nJ_m\otimes I_n )(\mathrm{vec}\, X) \\
  &\quad \ge 
\bigl| (\mathrm{vec}\, X)^T ( mI_m\otimes J_n - J_m\otimes J_n )(\mathrm{vec}\, X) \bigr|.
  		\end{aligned} \end{equation}   
Setting $A =J_m\otimes J_n$ in Theorem \ref{thm22} yields 
$$ mnI_{mn} -nJ_m\otimes I_n \ge 
\pm (mI_m\otimes J_n - J_m\otimes J_n),$$ 
and so (\ref{ck}) immediately follows. 
 \end{proof}

With the same method in the proof of Corollary \ref{coro24}, 
 the following corollary can be obtained from Corollary \ref{coro26},  
we omit the details of the proof. 

\begin{corollary} \label{coro25}
 Let $X=(x_{ij})$ be a real $m\times n$ matrix. Then
  		\begin{eqnarray*}    
mn\sum_{i=1}^m\sum_{j=1}^nx_{ij}^2 + n\sum_{j=1}^n\left(\sum_{i=1}^mx_{ij}\right)^2
\ge    m\sum_{i=1}^m\left(\sum_{j=1}^nx_{ij}\right)^2 + 
\left(\sum_{i=1}^m\sum_{j=1}^nx_{ij}\right)^2 .
  		\end{eqnarray*}
\end{corollary}

{\bf Remark.} 
Note that  $J_m\otimes J_n$ is not only 
a positive semidefinite matrix but also a PPT matrix, 
hence the weaker result  Proposition \ref{prop23} 
can also yields  Corollary \ref{coro25}.

\section{Appendix}

Motivated by the observation of Lin \cite[Proposition 2.2]{Lin16}, 
we next provide an alternative proof of Theorem \ref{thm22} 
by induction on the number of blocks of matrix. 
The following proof is  more transparent than that 
 in Section \ref{sec2}. 
We remark here that this proof has its root  in  \cite{Lin16} with slight differences.

\begin{proof}
The proof is by induction on $m$. 
Clearly, when $m=1$, there is nothing to show. 
Moreover the base case   $m=2$ was also proved in Section \ref{sec2}. Suppose the result (\ref{eqmain}) is true for $m=k-1>1$, 
and then we consider the case $m=k$, 
  		\begin{align*}  
\Gamma &:=(\tr A)I_{kn}+ I_k \otimes(\tr_1 A)-A-(\tr_2 A) \otimes I_n	\\
&= \left(\tr \sum_{i=1}^{k}A_{i,i}\right)I_{kn}
+I_k\otimes \left(\sum_{j=1}^k A_{j,j}\right)-A- 
\left( [\tr A_{i,j}]_{i,j=1}^k\right) \otimes I_n   		\\
&= \begin{bmatrix}  \sum_{i=1}^{k-1}(\tr A_{i,i})I_n &  &&  \\ &\ddots  & \\& &   \sum_{i=1}^{k-1}(\tr A_{i,i})I_n& \\ & & &  0 \end{bmatrix} \\
&\quad  + \begin{bmatrix} (\tr A_{k,k})I_n &  &&  \\ &\ddots  & \\& & (\tr A_{k,k})I_n & \\ & & &   \sum_{i=1}^k(\tr A_{i,i})I_n \end{bmatrix} 	\\ 
&\quad   + \begin{bmatrix}  \sum_{i=1}^{k-1}  A_{i,i} &  &&  \\ &\ddots  & \\& &   \sum_{i=1}^{k-1}  A_{i,i}& \\ & & &  0 \end{bmatrix}+ \begin{bmatrix}   A_{k,k} &  &&  \\ &\ddots  & \\& &   A_{k,k} & \\ & & &   \sum_{i=1}^k  A_{i,i} \end{bmatrix}	\\ 
& \quad  - \begin{bmatrix}   A_{1,1} & \cdots  & A_{1, k-1}& 0 \\ \vdots & & \vdots  & \vdots \\ A_{k-1,1}& \cdots &  A_{k-1, k-1}  & 0 \\ 0 & \cdots &  0&  0 \end{bmatrix}-\begin{bmatrix}  0 & \cdots  & 0&  A_{1, k} \\ \vdots & & \vdots  & \vdots \\0& \cdots &  0  & A_{k-1,k}  \\ A_{k,1} & \cdots &  A_{k,k-1} &  A_{k,k} \end{bmatrix}   	\\
&\quad   - \begin{bmatrix}  (\tr A_{1,1})I_n & \cdots  & (\tr A_{1, k-1})I_n& 0 \\ 
 \vdots & & \vdots  & \vdots \\ (\tr A_{k-1,1})I_n& \cdots &  (\tr A_{k-1, k-1})I_n  & 0 \\ 0 & \cdots &  0&  0 \end{bmatrix} \\
&\quad  -\begin{bmatrix}  0 & \cdots  & 0& (\tr A_{1, k})I_n \\ 
  \vdots & & \vdots  & \vdots \\0& \cdots &  0  & (\tr A_{k-1,k})I_n  \\ 
  (\tr A_{k,1})I_n & \cdots &  (\tr A_{k,k-1})I_n &  (\tr A_{k,k})I_n \end{bmatrix}.
\end{align*} 	
  By rearranging the terms, we may write
\[\Gamma =\Gamma_1 +\Gamma_2, \]
  		where 	
\begin{align*}  
\Gamma_1&:=\begin{bmatrix}  \sum_{i=1}^{k-1}(\tr A_{i,i})I_n &  &&  \\ &\ddots  & \\& &   \sum_{i=1}^{k-1}(\tr A_{i,i})I_n& \\ & & &  0 \end{bmatrix}
  		  +  \begin{bmatrix}  \sum_{i=1}^{k-1}  A_{i,i} &  &&  \\ &\ddots  & \\& &   \sum_{i=1}^{k-1}  A_{i,i}& \\ & & &  0 \end{bmatrix} \\ 
& \phantom{:}- \begin{bmatrix}   A_{1,1} & \cdots  & A_{1, k-1}& 0 \\ \vdots & & \vdots  & \vdots \\ A_{k-1,1}& \cdots &  A_{k-1, k-1}  & 0 \\ 0 & \cdots &  0&  0 \end{bmatrix} - \begin{bmatrix}  (\tr A_{1,1})I_n & \cdots  & (\tr A_{1, k-1})I_n& 0 \\ \vdots & & \vdots  & \vdots \\ (\tr A_{k-1,1})I_n& \cdots &  (\tr A_{k-1, k-1})I_n  & 0 \\ 0 & \cdots &  0&  0 \end{bmatrix},
  		\end{align*} 
and	
\begin{align*}  
\Gamma_2&:= \begin{bmatrix} (\tr A_{k,k})I_n &  &&  \\ &\ddots  & \\& & (\tr A_{k,k})I_n & \\ & & &   \sum_{i=1}^k(\tr A_{i,i})I_n \end{bmatrix}
  			+ \begin{bmatrix}   A_{k,k} &  &&  \\ &\ddots  & \\& &   A_{k,k} & \\ & & &   \sum_{i=1}^k  A_{i,i} \end{bmatrix}  \\
& \phantom{:} - \begin{bmatrix}  0 & \cdots  & 0&  A_{1, k} \\ \vdots & & \vdots  & \vdots \\0& \cdots &  0  & A_{k-1,k}  \\ A_{k,1} & \cdots &  A_{k,k-1} &  A_{k,k} \end{bmatrix}	 - 
\begin{bmatrix}  0 & \cdots  & 0& (\tr A_{1, k})I_n \\ 
 \vdots & & \vdots  & \vdots \\ 
0& \cdots &  0  & (\tr A_{k-1,k})I_n  \\ 
(\tr A_{k,1})I_n & \cdots &  (\tr A_{k,k-1})I_n &  (\tr A_{k,k})I_n \end{bmatrix} 		\\
&\phantom{:}= \begin{bmatrix}  (\tr A_{k,k})I_n+A_{k,k}  &   &  & -A_{1,k}-(\tr A_{1, k})I_n \\     
   & \ddots  &  & \vdots  \\  &  &  (\tr A_{k,k})I_n+A_{k,k}   &  -A_{k-1,k}-(\tr A_{k-1,k})I_n  \\ 
  -A_{k,1}-(\tr A_{k,1})I_n & \cdots &  -A_{k,k-1}-(\tr A_{k,k-1})I_n &  
   \sum_{i=1}^{k-1}\big((\tr A_{i,i})I_n+A_{i,i}\big)  \end{bmatrix}.
\end{align*} 	
  		Now by induction hypothesis, we get that $\Gamma_1$ is positive semidefinite. 
It remains to show that   $\Gamma_2$ is also positive semidefinite. 
  		
Observing that  $\Gamma_2$ can be written as a sum of $k-1$  
 matrices, in which each summand is $*$-congruent to 
$$H_i:=\begin{bmatrix}  (\tr A_{k,k})I_n+A_{k,k}  &    -A_{i,k}-(\tr A_{i, k})I_n \\   
 - A_{k,i}-(\tr A_{k,i})I_n &  (\tr A_{i,i})I_n+A_{i,i} \end{bmatrix}, \quad 
  		 i=1, 2,\ldots, k-1.$$  
Just like the proof of the base case, 
we infer from Lemma \ref{lem21} that $H_i\ge 0$ for all $i=1,2, \ldots, k-1$. 
Therefore, $\Gamma_2 \ge 0$, thus the proof of induction step is complete.  
  	   \end{proof}

\section*{Acknowledgments}
All authors would like to express sincere thanks to Prof. Tsuyoshi Ando for sharing \cite{Ando14} 
before its publication. 
The first author would like to express his hearty gratitude to 
Prof. Minghua Lin and Prof. Xiaohui Fu for detailed comments and  constant encouragement. 
This work was supported by  NSFC (Grant Nos. 11671402 and 11931002),  
Hunan Provincial Natural Science Foundation (Grant Nos. 2016JJ2138 and 2018JJ2479) 
and  Mathematics and Interdisciplinary Sciences Project of CSU.

\end{document}